\documentclass[a4paper,11pt,leqno]{amsart}
\usepackage[cp1250]{inputenc}
\usepackage{ifthen}
\usepackage{amssymb}
\usepackage{amsmath}
\usepackage{amsfonts}
\usepackage{latexsym}
\numberwithin{equation}{section}
\newtheorem{theo}{Theorem}[section]
\newtheorem{lem}{Lemma}[section]
\newtheorem{cor}{Corollary}[section]

\newtheorem{rem}{Remark}[section]

\begin{document}
\title[On partial sums]
{On partial sums of normalized Mittag-Leffler functions}
\date{}
\author{ Dorina R\u aducanu}
\address{Faculty of Mathematics and Computer Science, "Transilvania" University of Bra\c sov 50091, Iuliu Maniu, 50, Bra\c sov, Romania, e-mail:\textit{ draducanu@unitbv.ro }}


\begin{abstract}
This article deals with the ratio of normalized Mittag-Leffler function $\mathbb{E}_{\alpha,\beta}(z)$ and its sequence of partial sums $(\mathbb{E}_{\alpha,\beta})_m(z)$. Several examples which illustrate the validity of our results are also given.
\end{abstract}
\maketitle

\section{Introduction}
Let $\mathcal{A}$ be the class of functions $f$ normalized by 
\begin{equation}\label{1.1}
f(z)=z+\sum_{n=2}^\infty a_n z^n
\end{equation}
which are analytic in the open unit disk $\mathcal{U}=\left\{z\in\mathbb{C}: |z|<1\right\}$.

Denote by $\mathcal{S}$ the subclass of $\mathcal{A}$ which consists of univalent functions in $\mathcal{U}$.

Consider the function $E_\alpha(z)$ defined by 
\begin{equation}\label{1.2}
E_\alpha(z)=\sum_{n=0}^\infty\frac{z^n}{\Gamma(\alpha n+1)},\;\alpha>0,\;z\in\mathcal{U},
\end{equation}
where $\Gamma(s)$ denotes the familiar Gamma function. This function was introduced by Mittag-Leffler in $1903$ \cite{mil} and is therefore known as the Mittag-Leffler function.

Another function $E_{\alpha,\beta}(z)$, having similar properties to those of Mittag-Leffler function, was introduced by Wiman \cite{wim}, \cite{wima} and is defined by 
\begin{equation}\label{1.3}
E_{\alpha,\beta}(z)=\sum_{n=0}^\infty\frac{z^n}{\Gamma(\alpha n+\beta)},\;\alpha>0,\;\beta>0,\;z\in\mathcal{U}.
\end{equation}

During the last years the interest in Mittag-Leffler type functions has considerably increased due to their vast potential of applications in applied problems such as fluid flow, electric networks, probability, statistical distribution theory etc. For a detailed account of properties, generalizations and applications of functions (\ref{1.2}) - (\ref{1.3}) one may refer to \cite{gmr}, \cite{gud}, \cite{pra}, \cite{shp}.

Geometric properties including starlikeness, convexity and close-to-convexity for the Mittag-Leffler function $E_{\alpha,\beta}(z)$ were recently investigated by Bansal and Prajapat in \cite{bap}. Differential subordination results associated with generalized Mittag-Leffler function were also obtained in \cite{dor}.

The function defined by (\ref{1.3}) does not belong to the class $\mathcal{A}$. Therefore, we consider the following normalization of the Mittag-Leffler function $E_{\alpha,\beta}(z)$:
\begin{equation}\label{1.4}
\mathbb{E}_{\alpha,\beta}(z)=\Gamma(\beta)zE_{\alpha,\beta}(z)=z+\sum_{n=1}^\infty\frac{\Gamma(\beta)}{\Gamma(\alpha n+\beta)}z^{n+1},\;\alpha>0,\;\beta>0,\; z\in\mathcal{U}.
\end{equation}

Note that some special cases of $\mathbb{E}_{\alpha,\beta}(z)$ are:
\begin{equation}\label{1.5}
\left\{\begin{array}{ll}
\mathbb{E}_{2,1}(z)=z\cosh{\sqrt{z}}\\
\mathbb{E}_{2,2}(z)=\sqrt{z}\sinh(\sqrt{z}) \\
\mathbb{E}_{2,3}(z)=2[\cosh(\sqrt{z})-1] \\
\mathbb{E}_{2,4}(z)=6[\sinh(\sqrt{z})-\sqrt{z}]/\sqrt{z}.
\end{array}
\right.
\end{equation}

Recently, several results related to partial sums of special functions, such as Bessel \cite{ory}, Struve \cite{yao}, Lommel 
\cite{mer} and Wright functions \cite{mmy} were obtained.

Motivated by the work of Bansal and Prajapat \cite{bap} and also by the above mentioned results, in this paper we investigate the ratio of normalized Mittag-Leffler function $\mathbb{E}_{\alpha,\beta}(z)$ defined by (\ref{1.4}) to its sequence of partial sums
\begin{equation}\label{1.6}
\left\{\begin{array}{ll}
(\mathbb{E}_{\alpha,\beta})_0(z)=z\\
(\mathbb{E}_{\alpha,\beta})_m(z)=z+\displaystyle\sum_{n=1}^m A_nz^{n+1},\;m\in\mathbb{N}=\left\{1,2,\ldots\right\},
\end{array}
\right.
\end{equation}
where
\begin{equation*}
A_n=\frac{\Gamma(\beta)}{\Gamma(\alpha n+\beta)},\;\alpha>0,\;\beta>0,\;n\in\mathbb{N}.
\end{equation*}

We obtain lower bounds on ratios like 
\begin{equation*}
\Re\left\{\frac{\mathbb{E}_{\alpha,\beta}(z)}{(\mathbb{E}_{\alpha,\beta})_m(z)}\right\},\;\Re\left\{\frac{(\mathbb{E}_{\alpha,\beta})_m(z)}{\mathbb{E}_{\alpha,\beta}(z)}\right\},\;\Re\left\{\frac{\mathbb{E}'_{\alpha,\beta}(z)}{(\mathbb{E}_{\alpha,\beta})'_m(z)}\right\},\;\Re\left\{\frac{(\mathbb{E}_{\alpha,\beta})'_m(z)}{\mathbb{E}'_{\alpha,\beta}(z)}\right\}.
\end{equation*}
Several examples will be also given.

Results concerning partial sums of analytic functions may be found in \cite{fra}, \cite{lio}, \cite{oss}, \cite{rav}, \cite{sil}, \cite{siv} etc.

\section{Main results}

In order to obtain our results we need the following lemma.
\begin{lem}\label{l2.1}
Let $\alpha\geq1$ and $\beta\geq1$. Then the function $\mathbb{E}_{\alpha,\beta}(z)$ satisfies the next two inequalities:
\begin{equation}\label{2.1}
|\mathbb{E}_{\alpha,\beta}(z)|\leq\frac{\beta^2+\beta+1}{\beta^2},\; z\in\mathcal{U}
\end{equation}
\begin{equation}\label{2.2}
|\mathbb{E}'_{\alpha,\beta}(z)|\leq\frac{\beta^2+3\beta+2}{\beta^2},\; z\in\mathcal{U}.
\end{equation}
\end{lem}
\begin{proof}
Under the hypothesis we have $\Gamma(n+\beta)\leq\Gamma(\alpha n+\beta)$ and thus 
\begin{equation}\label{2.3}
\frac{\Gamma(\beta)}{\Gamma(\alpha n+\beta)}\leq\frac{\Gamma(\beta)}{\Gamma(n+\beta)}=\frac{1}{(\beta)_n},\; n\in\mathbb{N},
\end{equation}
where
\[
(x)_n=\left\{\begin{array}{ll}
1 & \mbox{, $n=0$}\\
x(x+1)\ldots(x+n-1) & \mbox{,$n\in\mathbb{N}$ }
\end{array}
\right.
\]
is the well-known Pochhammer symbol.

Note that 
\begin{equation}\label{2.4}
(x)_n=x(x+1)_{n-1},\;\;n\in\mathbb{N}
\end{equation}
and
\begin{equation}\label{2.5}
(x)_n\geq x^n,\;\;n\in\mathbb{N}.
\end{equation}
Making use of (\ref{2.3}) - (\ref{2.5}) and also of the well-known triangle inequality, for $z\in\mathcal{U}$, we obtain
\begin{equation*}
|\mathbb{E}_{\alpha,\beta}(z)|=\left|z+\sum_{n=1}^\infty\frac{\Gamma(\beta)}{\Gamma(\alpha n+\beta)}z^{n+1}\right|\leq 1+\sum_{n=1}^\infty\frac{\Gamma(\beta)}{\Gamma(\alpha n+\beta)}\leq 1+\sum_{n=1}^\infty\frac{1}{(\beta)_n}
\end{equation*}
\begin{equation*}
=1+\frac{1}{\beta}\sum_{n=1}^\infty\frac{1}{(\beta+1)_{n-1}}\leq 1+\frac{1}{\beta}\sum_{n=1}^\infty\frac{1}{(\beta+1)^{n-1}}=1+\frac{1}{\beta}\sum_{n=0}^\infty\left(\frac{1}{\beta+1}\right)^n=\frac{\beta^2+\beta+1}{\beta^2}
\end{equation*}
and thus, inequality (\ref{2.1}) is proved.

Using once more the triangle inequality, for $z\in\mathcal{U}$, we obtain
\begin{equation}\label{2.6}
|\mathbb{E}'_{\alpha,\beta}(z)|=\left|1+\sum_{n=1}^\infty\frac{(n+1)\Gamma(\beta)}{\Gamma(\alpha n+\beta)}z^n\right|\leq 1+\sum_{n=1}^\infty\frac{n\Gamma(\beta)}{\Gamma(\alpha n+\beta)}+\sum_{n=1}^\infty\frac{\Gamma(\beta)}{\Gamma(\alpha n+\beta)}.
\end{equation}
For $\beta\geq1$ we have
\begin{equation}\label{2.7}
\frac{n}{(\beta)_n}=\frac{n}{\beta(\beta+1)_{n-1}}=\frac{n}{\beta(\beta+1)_{n-2}(\beta+n-1)}\leq\frac{1}{\beta(\beta+1)_{n-2}}.
\end{equation}
Taking into account inequalities (\ref{2.3}) - (\ref{2.5}) and (\ref{2.7}), from (\ref{2.6}), we obtain 
\begin{equation*}
|\mathbb{E}'_{\alpha,\beta}(z)|\leq 1+\sum_{n=1}^\infty\frac{n}{(\beta)_n}+\sum_{n=1}^\infty\frac{1}{(\beta)_n}\leq 1+\frac{1}{\beta}+\frac{1}{\beta}\sum_{n=2}^\infty\frac{1}{(\beta+1)_{n-2}}+\frac{1}{\beta}\sum_{n=1}^\infty\frac{1}{(\beta+1)_{n-1}}
\end{equation*}
\begin{equation*}
\leq 1+\frac{1}{\beta}+\frac{1}{\beta}\sum_{n=0}^\infty\left(\frac{1}{\beta+1}\right)^n+\frac{1}{\beta}\sum_{n=0}^\infty\left(\frac{1}{\beta+1}\right)^n=\frac{\beta^2+3\beta+2}{\beta^2}
\end{equation*}
and thus, inequality (\ref{2.2}) is also proved.
\end{proof}

Let $w(z)$ be an analytic function in $\mathcal{U}$. In the sequel, we will frequently use the following well-known result:
\begin{equation*}
\Re\left\{\frac{1+w(z)}{1-w(z)}\right\}>0,\;z\in\mathcal{U}\;\;\textrm{if and only if}\;\;|w(z)|<1,\;z\in\mathcal{U}.
\end{equation*}
\begin{theo}\label{t2.1}
Let $\alpha\geq1$ and $\beta\geq\displaystyle\frac{1+\sqrt{5}}{2}$. Then
\begin{equation}\label{2.8}
\Re\left\{\frac{\mathbb{E}_{\alpha,\beta}(z)}{(\mathbb{E}_{\alpha,\beta})_m(z)}\right\}\geq\frac{\beta^2-\beta-1}{\beta^2},\;z\in\mathcal{U}
\end{equation}
and
\begin{equation}\label{2.9}
\Re\left\{\frac{(\mathbb{E}_{\alpha,\beta})_m(z)}{\mathbb{E}_{\alpha,\beta}(z)}\right\}\geq\frac{\beta^2}{\beta^2+\beta+1},\;z\in\mathcal{U}.
\end{equation}
\end{theo}
\begin{proof}
From inequality (\ref{2.1}) we get
\begin{equation*}
1+\sum_{n=1}^\infty A_n\leq\frac{\beta^2+\beta+1}{\beta^2},\;\;\textrm{where}\;\;A_n=\frac{\Gamma(\beta)}{\Gamma(\alpha n+\beta)},\;\;n\in\mathbb{N}.
\end{equation*}
The last inequality is equivalent to
\begin{equation*}
\frac{\beta^2}{\beta+1}\sum_{n=1}^\infty A_n\leq1.
\end{equation*}
In order to prove the inequality (\ref{2.8}), we consider the function $w(z)$ defined by 
\begin{equation*}
\frac{1+w(z)}{1-w(z)}=\frac{\beta^2}{\beta+1}\frac{\mathbb{E}_{\alpha,\beta}(z)}{(\mathbb{E}_{\alpha,\beta})_m(z)}-\frac{\beta^2-\beta-1}{\beta+1}
\end{equation*}
or
\begin{equation}\label{2.10}
\frac{1+w(z)}{1-w(z)}=\frac{1+\displaystyle\sum_{n=1}^m A_nz^n+\frac{\beta^2}{\beta+1}\sum_{n=m+1}^\infty A_nz^n}{1+\displaystyle\sum_{n=1}^m A_nz^n}.
\end{equation}
From (\ref{2.10}), we obtain 
\begin{equation*}
w(z)=\frac{\displaystyle\frac{\beta^2}{\beta+1}\sum_{n=m+1}^\infty A_nz^n}{2+2\displaystyle\sum_{n=1}^m A_nz^n+\frac{\beta^2}{\beta+1}\sum_{n=m+1}^\infty A_nz^n}
\end{equation*}
and
\begin{equation*}
|w(z)|<\frac{\displaystyle\frac{\beta^2}{\beta+1}\sum_{n=m+1}^\infty A_n}{2-2\displaystyle\sum_{n=1}^m A_n-\frac{\beta^2}{\beta+1}\sum_{n=m+1}^\infty A_n}.
\end{equation*}
The inequality $|w(z)|<1$ holds true if and only if 
\begin{equation*}
\frac{2\beta^2}{\beta+1}\sum_{n=m+1}^\infty A_n\leq 2-2\sum_{n=1}^m A_n
\end{equation*}
which is equivalent to 
\begin{equation}\label{2.11}
\sum_{n=1}^m A_n+\frac{\beta^2}{\beta+1}\sum_{n=m+1}^\infty A_n\leq1.
\end{equation}
To prove (\ref{2.11}), it suffices to show that its left-hand side is bounded above by 
\begin{equation*}
\frac{\beta^2}{\beta+1}\sum_{n=1}^\infty A_n
\end{equation*}
which is equivalent to 
\begin{equation*}
\frac{\beta^2-\beta-1}{\beta+1}\sum_{n=1}^m A_n\geq 0.
\end{equation*}
The last inequality holds true for  $\beta\geq\displaystyle\frac{1+\sqrt{5}}{2}$.

We use the same method to prove inequality (\ref{2.9}). Consider the function $w(z)$ given by
\begin{equation*}
\frac{1+w(z)}{1-w(z)}=\frac{\beta^2+\beta+1}{\beta+1}\frac{(\mathbb{E}_{\alpha,\beta})_m(z)}{\mathbb{E}_{\alpha,\beta}(z)}-\frac{\beta^2}{\beta+1}.
\end{equation*}
From the last equality we obtain
\begin{equation*}
w(z)=\frac{-\displaystyle\frac{\beta^2+\beta+1}{\beta+1}\sum_{n=m+1}^\infty A_nz^n}{2+2\displaystyle\sum_{n=1}^m A_nz^n-\frac{\beta^2-\beta-1}{\beta+1}\sum_{n=m+1}^\infty A_nz^n}
\end{equation*}
and 
\begin{equation*}
|w(z)|<\frac{\displaystyle\frac{\beta^2+\beta+1}{\beta+1}\sum_{n=m+1}^\infty A_n}{2-2\displaystyle\sum_{n=1}^m A_n-\frac{\beta^2-\beta-1}{\beta+1}\sum_{n=m+1}^\infty A_n}.
\end{equation*}
Then, $|w(z)|<1$ if and only if 
\begin{equation}\label{2.12}
\frac{\beta^2}{\beta+1}\sum_{n=m+1}^\infty A_n+\sum_{n=1}^m A_n\leq 1.
\end{equation}
Since the left-hand side of (\ref{2.12}) is bounded above by
\begin{equation*}
\frac{\beta^2}{\beta+1}\sum_{n=1}^\infty A_n
\end{equation*}
we have that the inequality (\ref{2.9}) holds true.
Now, the proof of our theorem is completed. 
\end{proof}

In the next theorem we consider ratios involving derivatives.
\begin{theo}\label{2.2}
Let $\alpha\geq 1$ and let $\beta\geq\displaystyle\frac{3+\sqrt{17}}{2}$. Then
\begin{equation}\label{2.13}
\Re\left\{\frac{\mathbb{E}'_{\alpha,\beta}(z)}{(\mathbb{E}_{\alpha,\beta})'_m(z)}\right\}\geq\frac{\beta^2-3\beta-2}{\beta^2},\;z\in\mathcal{U}
\end{equation}
and
\begin{equation}\label{2.14}
\Re\left\{\frac{(\mathbb{E}_{\alpha,\beta})'_m(z)}{\mathbb{E}'_{\alpha,\beta}(z)}\right\}\geq\frac{\beta^2}{\beta^2+3\beta+2},\;z\in\mathcal{U}.
\end{equation}
\end{theo}
\begin{proof}
From (\ref{2.2}) we have
\begin{equation*}
1+\sum_{n=1}^\infty (n+1)A_n\leq\frac{\beta^2+3\beta+2}{\beta^2}, \;\;\textrm{where}\;\;A_n=\frac{\Gamma(\beta)}{\Gamma(\alpha n+\beta)},\;n\in\mathbb{N}.
\end{equation*}
The above inequality is equivalent to
\begin{equation*}
\frac{\beta^2}{3\beta+2}\sum_{n=1}^\infty (n+1)A_n\leq 1.
\end{equation*}
To prove (\ref{2.13}), define the function $w(z)$ by
\begin{equation*}
\frac{1+w(z)}{1-w(z)}=\frac{\beta^2}{3\beta+2}\frac{\mathbb{E}'_{\alpha,\beta}(z)}{(\mathbb{E}_{\alpha,\beta})'_m(z)}-\frac{\beta^2-3\beta-2}{3\beta+2}
\end{equation*}
which gives
\begin{equation*}
w(z)=\frac{\displaystyle\frac{\beta^2}{3\beta+2}\sum_{n=m+1}^\infty(n+1)A_nz^n}{2+2\displaystyle\sum_{n=1}^m (n+1)A_nz^n+\frac{\beta^2}{3\beta+2}\sum_{n=m+1}^\infty (n+1)A_nz^n}
\end{equation*}
and
\begin{equation*}
|w(z)|<\frac{\displaystyle\frac{\beta^2}{3\beta+2}\sum_{n=m+1}^\infty(n+1)A_n}{2-2\displaystyle\sum_{n=1}^m (n+1)A_n-\frac{\beta^2}{3\beta+2}\sum_{n=m+1}^\infty (n+1)A_n}.
\end{equation*}
The condition $|w(z)|<1$ holds true if and only if
\begin{equation}\label{2.15}
\sum_{n=1}^m (n+1)A_n+\frac{\beta^2}{3\beta+2}\sum_{n=m+1}^\infty(n+1)A_n\leq 1.
\end{equation}
The left-hand side of (\ref{2.15}) is bounded above by 
\begin{equation*}
\frac{\beta^2}{3\beta+2}\sum_{n=1}^\infty (n+1)A_n\;\;\;\textrm{if}\;\;\;\;\frac{\beta^2-3\beta-2}{3\beta+2}\sum_{n=1}^m (n+1)A_n\geq 0
\end{equation*}
which holds true for $\beta\geq\displaystyle\frac{3+\sqrt{17}}{2}$.

The proof of (\ref{2.14}) follows the same pattern. Consider the function $w(z)$ given by
\begin{equation*}
\frac{1+w(z)}{1-w(z)}=\frac{\beta^2+3\beta+2}{3\beta+2}\frac{(\mathbb{E}_{\alpha,\beta})'_m(z)}{\mathbb{E}'_{\alpha,\beta}(z)}-\frac{\beta^2}{3\beta+2}
\end{equation*}
\begin{equation}\label{2.16}
=\frac{1+\displaystyle\sum_{n=1}^m(n+1)A_nz^n-\frac{\beta^2}{3\beta+2}\sum_{n=m+1}^\infty (n+1)A_nz^n}{1+\displaystyle\sum_{n=1}^\infty(n+1)A_nz^n}.
\end{equation}
From (\ref{2.16}), we can write 
\begin{equation*}
w(z)=\frac{-\displaystyle\frac{\beta^2+3\beta+2}{3\beta+2}\sum_{n=m+1}^\infty(n+1)A_nz^n}{2+2\displaystyle\sum_{n=1}^m (n+1)A_nz^n-\frac{\beta^2-3\beta-2}{3\beta+2}\sum_{n=m+1}^\infty (n+1)A_nz^n}
\end{equation*}
and
\begin{equation*}
|w(z)|<\frac{\displaystyle\frac{\beta^2+3\beta+2}{3\beta+2}\sum_{n=m+1}^\infty(n+1)A_n}{2-2\displaystyle\sum_{n=1}^m (n+1)A_n-\frac{\beta^2-3\beta-2}{3\beta+2}\sum_{n=m+1}^\infty (n+1)A_n}.
\end{equation*}
The last inequality implies that $|w(z)|< 1$ if and only if 
\begin{equation*}
\frac{2\beta^2}{3\beta+2}\sum_{n=m+1}^\infty(n+1)A_n\leq 2-2\sum_{n=1}^m (n+1)A_n
\end{equation*}
or equivalently
\begin{equation}\label{2.17}
\sum_{n=1}^m (n+1)A_n+\frac{\beta^2}{3\beta+2}\sum_{n=m+1}^\infty(n+1)A_n\leq 1.
\end{equation}
It remains to show that the left-hand side of (\ref{2.17}) is bounded above by
\begin{equation*}
\frac{\beta^2}{3\beta+2}\sum_{n=1}^\infty(n+1)A_n.
\end{equation*}
This is equivalent to
\begin{equation*}
\frac{\beta^2-3\beta-2}{3\beta+2}\sum_{n=1}^m(n+1)A_n\geq 0\;\;\textrm{which holds true for}\;\;\beta\geq\frac{3+\sqrt{17}}{2}.
\end{equation*}
 Now, the proof of our theorem is completed.
\end{proof}

\section{Examples}

In this section we give several examples which illustrate our theorems.

A result involving the functions $\mathbb{E}_{2,2}(z)$ and $\mathbb{E}_{2,3}(z)$, defined by (\ref{1.5}), can be obtained from Theorem \ref{t2.1} by taking $m=0, \alpha=2, \beta=2$ and $m=0, \alpha=2, \beta=3$, respectively.
\begin{cor}\label{c3.1}
The following inequalities hold true:
\begin{equation*}
\Re\left\{\frac{\sinh(\sqrt{z})}{\sqrt{z}}\right\}\geq\frac{1}{4}=0,25\;,\;\;\Re\left\{\frac{\sqrt{z}}{\sinh(\sqrt{z})}\right\}\geq\frac{4}{7}\approx 0,57
\end{equation*}
and
\begin{equation*}
\Re\left\{\frac{\cosh(\sqrt{z})-1}{z}\right\}\geq\frac{5}{18}\approx 0,28\;,\;\;\Re\left\{\frac{z}{\cosh(\sqrt{z})-1}\right\}\geq\frac{18}{13}\approx 1,38.
\end{equation*}
\end{cor}

Setting $m=0, \alpha=2$ and $\beta=4$ in Theorem \ref{2.1} and Theorem \ref{2.2} respectively, we obtain the next result involving the function $\mathbb{E}_{2,4}(z)$, defined by (\ref{1.5}), and its derivative.
\begin{cor}\label{c3.2}
The following inequalities hold true:
\begin{equation*}
\Re\left\{\frac{\sinh(\sqrt{z})-\sqrt{z}}{z\sqrt{z}}\right\}\geq\frac{11}{96}\approx0,11\;,\;\;\Re\left\{\frac{z\sqrt{z}}{\sinh(\sqrt{z})-\sqrt{z}}\right\}\geq\frac{32}{7}\approx 4,57
\end{equation*}
and
\begin{equation*}
\Re\left\{\frac{\sqrt{z}\cosh(\sqrt{z})-\sinh(\sqrt{z})}{z\sqrt{z}}\right\}\geq\frac{1}{24}\approx 0,04\;,\;\;\Re\left\{\frac{z\sqrt{z}}{\sqrt{z}\cosh(\sqrt{z})-\sinh(\sqrt{z})}\right\}\geq\frac{8}{5}=1,6.
\end{equation*}
\end{cor}
\begin{rem}
If we consider $m=0$ in inequality (\ref{2.13}), we obtain $\Re\left\{\mathbb{E}'_{\alpha,\beta}(z)\right\}>0$. In view of Noshiro-Warschawski Theorem (see \cite{goo}), we have that the normalized Mittag-Leffler function is univalent in $\mathcal{U}$ for 
$\alpha\geq 1$ and $\beta\geq\displaystyle\frac{3+\sqrt{17}}{2}$.
\end{rem}

\end{document}